\documentclass[12pt]{amsart}

\usepackage{comment}
\newcommand{\disp}{\displaystyle}
\newcommand{\nc}{\newcommand}
\nc{\G}{{\Gamma}} \nc{\BC}{{\mathbb C}} \nc{\BQ}{{\mathbb Q}}
\nc{\BR}{{\mathbb R}} \nc{\BZ}{{\mathbb Z}} \nc{\BP}{{\mathbb P}} \nc{\PC}{{\BP_1(\BC)}}
\nc{\BN}{{\mathbb N}} \nc{\BM}{{\mathbb M}}
\nc{\fH}{{\mathbb H}}

\nc{\mat}{{\binom{a\,\ b}{c\,\ d}}}
\nc{\U}{{\mathcal U}}
\nc{\PS}{{\mbox{PSL}_2(\BZ)}} \nc{\SL}{{\mbox{SL}_2(\BZ)}}
\nc{\SR}{{\mbox{SL}_2(\BR)}} \nc{\PR}{{\mbox{PSL}_2(\BR)}}
\nc{\SLC}{{\mbox{SL}_2(\BC)}}

\nc{\GL}{{\mbox{GL}}} \nc{\PQ}{{\mbox{PGL}_2^+(\BQ)}}
\nc{\GR}{{\mbox{GL}_2^+(\BR)}} \nc{\PG}{{\mbox{PGL}_2(\BC)}}
\nc{\GC}{{\mbox{GL}_2(\BC)}}
\nc{\f}{{\mathcal{F}(\fH)}}
\nc{\Cc}{\widehat{\BC}}
\nc{\e}{{E_{\varrho}(\G)}}
\nc{\g}{{\gamma}}
\nc{\vm}{{V_{\varrho}(\G)}}
\nc{\oo}{{\mathcal O}}
\nc{\M}{{\mbox{M}}}
\nc{\om}{{\omega}}
\nc{\Om}{{\Omega}}
\nc{\TX}{{\widetilde{X}}}
\nc{\ol}{\overline}
\nc{\cl}{{\mathcal L}}
\nc{\ce}{{\mathcal E}}
\nc{\la}{{\lambda}}
\nc{\La}{{\Lambda}}

\nc{\cz}{{\mathcal Z}}

\newtheorem{numbered}{}[section]
\newtheorem{thm}[numbered]{Theorem}

\newtheorem{definition}[numbered]{Definition}

\newtheorem{lem}[numbered]{Lemma}
\newtheorem{remark}[numbered]{Remark}
\newtheorem{prop}[numbered]{Proposition}

\numberwithin{equation}{section}

\newcommand{\thmref}[1]{Theorem~\ref{#1}}
\newcommand{\propref}[1]{Proposition~\ref{#1}}

\newcommand{\lemref}[1]{Lemma~\ref{#1}}

\usepackage{amsmath,amsfonts,amscd,amssymb,amsthm,comment}
\usepackage{booktabs}
\usepackage{color}
\usepackage{tikz-cd}
\usepackage{multirow}
\usepackage{array}
\usepackage{graphicx}

\usepackage{tabularx}
\usepackage{enumitem}
\usepackage{tikz-cd} 
\setitemize{itemsep=.5mm,topsep=1.5mm,parsep=0pt,
partopsep=0pt}

\begin{document}
	
	\title[]{On level 2 Modular differential equations}
	\author[]{Khalil Besrour} \author[]{Abdellah Sebbar}

	\address{Department of Mathematics and Statistics, University of Ottawa,
		Ottawa Ontario K1N 6N5 Canada}
	\email{kbesr067@uottawa.ca}
	\email{asebbar@uottawa.ca}

\begin{abstract}
In this paper, we explore the modular differential equation $\disp y'' + F(z)y = 0$ on the upper half-plane $\mathbb{H}$, where $F$ is a weight 4 modular form for $\Gamma_0(2)$. Our approach centers on solving the associated Schwarzian equation $\displaystyle \{h, z\} = 2F(z)$, where $\{h, z\}$ represents the Schwarzian derivative of a meromorphic function $h$ on $\mathbb{H}$. We derive conditions under which the solutions to this equation are modular functions for subgroups of the modular group and provide explicit expressions for these solutions in terms of classical modular functions. Key tools in our analysis include the theory of equivariant functions on the upper half-plane and the representation theory of level 2 subgroups of the modular group. The corresponding case where $F$ is a modular form for the full modular group $\mbox{SL}_2(\BZ)$ has been previously addressed in \cite{forum}.
\end{abstract}

\maketitle

\section{Introduction}

In a series of papers \cite{forum,jmaa,ramanujan,rev,ramanujan2}, the authors examined the differential equation
\begin{equation}\label{eqn1}
    y'' + s\, E_4\, y = 0,
\end{equation}
defined over the complex upper half-plane $\fH$. Here, $E_4$ is the weight 4 Eisenstein series and $s$ is a complex parameter.
Their approach involved analyzing a closely related equation:
\begin{equation}\label{eqn2}
     \{h, z\} = s\, E_4,
\end{equation}
where \( \{h, z\} \) represents the Schwarzian derivative of \( h \), defined for a meromorphic function \( h \) on a domain in \( \mathbb{C} \) as follows:
\begin{equation}\label{schwarz}
    \{h, z\} := \left( \frac{h''}{h'} \right)' - \frac{1}{2} \left( \frac{h''}{h'} \right)^2.
\end{equation}
There is a well-defined correspondence between the solutions of \eqref{eqn1} and those of \eqref{eqn2}. Specifically, if \( y_1 \) and \( y_2 \) are linearly independent meromorphic solutions to \eqref{eqn1}, then  \( y_2/y_1 \) is a solution to  \eqref{eqn2}. Conversely, given a solution \( h \) to \eqref{eqn2}, the functions \(y_1= 1/{\sqrt{h'}} \) and \( y_2={h}/{\sqrt{h'}} \) are well-defined meromorphic functions of $\fH$ and form a pair of  linearly independent solutions to \eqref{eqn1}. In other words, solving one equation will lead to the solutions of the other.

In the meantime, if we are given a modular form $F$ of weight 4 for some subgroup $\G$ of the modular group, then for any solution $h$ to  
$$  \{ h, z \} = F,$$
 there exists a two-dimensional representation $\rho : G \to \text{GL}_2(\mathbb{C})$ such that, for every $\alpha \in \G$ and $\tau \in \fH$, we have
\[
    h(\alpha \cdot \tau) = \rho(\alpha) \cdot h(\tau),
\]
where the action on both sides is by linear fractional transformations.
In this case, $h$ is referred to as a $\rho$-equivariant function on $\G$.

When the representation $\rho$ associated with $h$ has a finite image, $h$ is a modular function for some finite-index subgroup of $\G$. In \cite{forum}, the authors considered the case where $\G$ is the full modular group, for which the space of weight 4 modular forms is 1-dimensional generated by  $E_4$, and established the following result:

\begin{thm}\cite{forum}\label{forum1}
    A solution to $\{ h, z\} = s \, E_4$ is a modular function for some finite index subgroup of $\SL$ if and only if $\disp s = 2 \pi^2 \left( \frac{n}{m} \right)^2$, where $n$ and $m$ are relatively prime positive integers, and $2 \leq m \leq 5$. Moreover, when the solution $h$ is a modular function, its invariance group is the principal congruence group $\G(m)$.
\end{thm}

An example of such a solution, when $n/m=1/2$, is the Klein $\lambda-$function, which satisfies
$$
\{ \lambda, \tau \} = \frac{\pi^2}{2} E_4 .
$$

In this paper, we focus on the level 2 congruence group $\Gamma_0(2)$ and study the corresponding modular differential equation with coefficients given by modular forms for $\Gamma_0(2)$. Specifically, we aim to solve the Schwarzian differential equation
$\{h,\tau\}=F(\tau)$,
where $F$ is a weight 4 modular form for $\Gamma_0(2)$. The space of such forms, denoted $M_4(\Gamma_0(2))$, is two-dimensional and has as a basis $\{ \theta_2^8 , (\theta_3 \theta_4)^4 \}$, where $\theta_2$, $\theta_3$, and $\theta_4$ are the Jacobi theta functions. The Schwarzian equation then becomes 
\begin{equation}\label{eqn3}  
\{h,\tau\} \,=\, a\, \theta_2^8 + b \, (\theta_3 \theta_4)^4, \
\end{equation} 
where $a, b \in \mathbb{C}$. While the choice of this specific basis for $M_4(\Gamma_0(2))$ may seem arbitrary, it simplifies the computations. A key feature of this basis is that each element vanishes at one of the two inequivalent cusps of $\Gamma_0(2)$.

In this paper, we will precisely determine the conditions under which \eqref{eqn3} admits solutions that are modular functions. Specifically, we will prove the following result:
\begin{thm}
A solution to the differential equation
\[
\left\{ h, z \right\} = a \, (\theta_2^8) + b \, (\theta_3 \theta_4)^4
\]
is a modular function if and only if
\[
a = 2 \pi^2 \left(\frac{n_1}{m_1}\right)^2 \quad \text{and} \quad b = 2 \pi^2 \left(\frac{n_2}{m_2}\right)^2,
\]
where $n_1, m_1$ (resp. $n_2, m_2$) are relatively prime positive integers, and one of the following holds:
\begin{enumerate}
    \item $a = b$ and $m_1 = m_2 \in \{2, 3, 4, 5\}$, \\
    \item $(m_1, m_2) \in \{(n, 4) \mid n \geq 1\} \cup \{(2, 2n) \mid n \geq 1\}$, \\
    \item $(m_1, m_2) \in \{(3, 6), (4, 6), (3, 8), (5, 6), (3, 10)\}$.
\end{enumerate}
\end{thm}

In each case, we will determine the finite index subgroup of the modular group on which the solution is modular. These subgroups are of genus 0, and in most cases, we can explicitly determine their Hauptmoduln.

Our investigation employs various tools, including the theory of equivariant functions, the representation theory of $\Gamma_0(2)$, and the Riemann-Hurwitz formula, among others.

The structure of the paper is as follows. Section 2 provides definitions and basic properties of the tools we will use, including a review of key properties of the modular group and the Schwarzian derivative. In Section 3, we connect the Schwarzian equation \eqref{eqn2} to equivariant functions and explain why, in our context, it suffices to consider projective representations. Section 4 applies the Frobenius method to the second-order modular differential equation \eqref{eqn1}, following a change of variables. We first approach this for a general congruence subgroup and then specialize to $\Gamma_0(2)$. Section 5 classifies all representations of $\Gamma_0(2)$ that can give rise to modular solutions. We also determine the invariance groups of these solutions and investigate their genus using the Riemann-Hurwitz formula. Finally, in Section 6, we explicitly compute some solutions in terms of classical modular functions.

\section{Preliminaries}

Let $\SL$ (or $\Gamma(1)$) be the group of $2\times2$ invertible matrices with integer entries. We denote by $\PS$ (or $\overline{\Gamma(1)}$)  the quotient $\SL/\{I,-I\}$ and, for every subgroup $\Gamma \leq \text{SL}_2(\mathbb{Z})$, we denote by $\overline{\Gamma}$ the image of $\Gamma$ via the projection $\SL \to \PS$. We refer to $\SL$ (or to $\PS$ if we are working projectively) as the modular group.

The principal congruence subgroup of level $m$, where $m$ is a positive integer, is defined by
$$ \Gamma(m) := \{ A \in \Gamma(1) \; \mid \; A \equiv I \; \text{ mod }m \}.$$
A subgroup of $\Gamma(1)$ is said to be congruence of level $m$ if it contains $\Gamma(m)$ and no other $\Gamma(n)$ for $n <m$. Examples of such groups are
$$ \Gamma_0(m) = \left\{ \begin{pmatrix}
a & b \\
c & d 
\end{pmatrix} \in \Gamma(1) \; \mid \; c \equiv 0 \; \text{ mod }m \right\}\ ,\ \ m\geq 1.$$

Any subgroup $\G$ of $\SL$ acts on the upper half plane $\mathbb{H} := \{ z \in \mathbb{C} \mid \text{Im}(z) > 0\}$ via linear fractional transformations: for any $\alpha=\begin{pmatrix}
a & b \\
c & d 
\end{pmatrix} \in \G$ and $\tau \in \mathbb{H}$, we have
$$ \alpha\cdot \tau := \frac{a \tau + b}{c \tau +d}.$$

\begin{definition} Let $\Gamma$ be a subgroup of $\SL$ and let $\rho$ be a two-dimensional representation of $\Gamma$. 
A $\rho$-equivariant function associated to $(\Gamma, \rho)$ is a meromorphic function $f$ on $\mathbb{H}$ such that $$f(\alpha \cdot\tau) = \rho(\alpha) \cdot f(\tau) \; \; \text{ for all } \alpha \in \G, \tau \in \mathbb{H}$$
where the action $\cdot$ on both sides is by linear fractional  transformation.
\end{definition}
From the definition, it follows that if $f$ is a $\rho-$equivariant function on $\Gamma$ and $-I \in \Gamma$, then $\rho(-I) \in \mathbb{C}^*I$. Hence, $\rho$ factors through $\{I,-I\}$ inducing a (projective) representation $\overline{\rho} : \overline{\Gamma} \to \PG$. For a more detailed exposition of the theory of $\rho-$equivariant functions and their connection to vector-valued modular forms, we refer the reader to \cite{vvmf}. Notice that this definition can be extended to any discrete subgroup of $\SR$.

Let 
$$ T = \left[ {\begin{array}{cc}
   1 & 1 \\
   0 & 1 \\
  \end{array} } \right], \; S = \left[ {\begin{array}{cc}
   0 & -1 \\
   1 & 0 \\
  \end{array} } \right] \; \text{ and } \; R = ST^{-2}S^{-1} = \left[ {\begin{array}{cc}
   1 & 0 \\
   2 & 1 \\
  \end{array} } \right].$$

The congruence subgroup $\Gamma_0(2)$ is generated by the matrices $-I, T$ and $R$ which verify the relation $(RT^{-1})^2= -I.$ Moreover, the group can be given by the following presentation 
$$ \left< a,b \mid b^4=1, \, ab^2=b^2a \right>.$$
Its projective image $\overline{\Gamma}_0(2)$  is isomorphic to the free product $\mathbb{Z} \ast \mathbb{Z}_2 := \left< a,b \mid b^2 =1\right>$ via the map 
\begin{equation}\label{isom}
    (a,b) \mapsto \left(T,RT^{-1}\right).
\end{equation}

We recall that $\Gamma_0(2)$ is a subgroup of index 3 in $\G(1)$ and has two inequivalent cusps. In this paper, we will use $0$ and $\infty$ as representatives of their respective equivalence classes.

Let $\Gamma$ be a subgroup of the modular group. A weakly modular form $h$ of weight $k$ over $\Gamma$ is a meromorphic function on $\mathbb{H}$ that verifies
$$ h\left(\begin{pmatrix}
a & b \\
c & d 
\end{pmatrix} \cdot \tau \right) = (c \tau + d)^k h(\tau) \; \; \text{ for all } \begin{pmatrix}
a & b \\
c & d 
\end{pmatrix} \in \Gamma. $$
A modular form $h$ is a weakly modular form that is also holomorphic on $\mathbb{H}$ and at the cusps.

We end this section by including some properties of the Schwarzian derivative, defined in \eqref{schwarz}, which we will need in the rest of the paper.

\begin{prop}\label{sch}
    Let $h$ and $f$ be meromorphic functions on $\mathbb{H}$, then
\begin{enumerate}
    \item $\{ h, \tau\} = \{ g, \tau\}$   \text{ if and only if }   $f= \gamma \cdot h$ \; for some $\gamma \in \text{GL}_2(\mathbb{C})$.

    \item For all  $\gamma \in \text{GL}_2(\mathbb{C})$, we have 
    $$ \{ f, \gamma \cdot \tau\} = \frac{j(\gamma, \tau)^4}{\text{{det}}^2(\gamma)} \, \{ f, \tau\},$$
    where $j\left( \begin{pmatrix}
a & b \\
c & d 
\end{pmatrix},\tau\right) := c \tau + d.$

    \item If $\omega$ is a meromorphic function of $\tau$ then
    \begin{equation}\label{cocycle} \{f, \tau\} = \{f, \omega\} \left( \frac{d\omega}{d\tau}\right)^2 + \{ \omega, \tau\}.
    \end{equation}
\end{enumerate}
\end{prop}

We end this section by introducing some classical modular forms and modular functions that we will use in the rest of the paper. The main reference here is \cite{rankin}.

For $\tau \in \mathbb{H},$ let $q = e^{2 \pi i \tau}$. The Dedekind $\eta-$function is defined by:
$$ \eta(\tau) = q^{\frac{1}{24}} \prod_{n=1}^\infty (1 - q^n).$$
Under the action of $T$ and $S$, it verifies
\begin{equation}
    \eta(\tau + 1) = e^{\frac{2 \pi i}{12}}\eta(\tau) \;  \text{ and } \; \eta\left(\frac{-1}{\tau}\right) = \sqrt{-i \tau} \eta(\tau).
\end{equation}
The Jacobi theta functions are given by:
$$\theta_2(\tau) = 2q^{\frac{1}{8}}\prod_{n=1}^\infty (1-q^n)(1+q^n)^2= \frac{2\eta(4\tau)^2}{\eta(2 \tau)},$$
$$\theta_3(\tau) = \prod_{n=1}^\infty (1-q^n)(1+q^{n-\frac{1}{2}})^2 = \frac{\eta(2\tau)^5}{\eta(\tau)^2 \eta(4 \tau)^2}$$
and 
$$\theta_4(\tau) = \prod_{n=1}^\infty (1-q^n)(1-q^{n-\frac{1}{2}})^2 = \frac{\eta(\tau)^2}{\eta(2 \tau)}.$$

Under $T$ and $S$, the Jacobi theta functions transform as follows:
$$ \Big(\theta_2, \theta_3, \theta_4\Big) \circ T = \Big( \sqrt{i}\theta_2, \theta_4, \theta_3\Big)$$
and
$$ \Big(\theta_2, \theta_3, \theta_4\Big) \circ S = \sqrt{-i \tau}\Big(\theta_4, \theta_3, \theta_2\Big).$$

The classical Eisenstein series of weight $4$ and $6$ are defined respectively as 
$$ E_4 (\tau)= 1 + 240 \sum_{n=1}^{\infty} \sigma_3(n) q^n$$
and
$$
E_6(\tau)=1-504\sum_{n=1}^{\infty} \sigma_5(n) q^n,
$$
where $\sigma_k(n)$ is the sum of the $k-$th powers of the positive divisors of $n$. We also recall that the elliptic modular function $\lambda$ is given by
$$ \lambda = \frac{\theta_2^4}{\theta_3^4}.$$
Finally, we present some identities that will be useful later on:
\begin{equation}\label{thetas}
    \theta_3^4 = \theta_2^4 + \theta_4^4 \; \;  \text{ and } \; \;  E_4 = \theta_2^8 + (\theta_3 \theta_4)^4. 
\end{equation}

\section{Solutions as $\rho$-equivariant functions.}
Unless stated otherwise, we will assume throughout this section that $\Gamma$ is a finite-index subgroup of $\Gamma(1)$.  Using the properties of the Schwarzian derivative discussed earlier, we derive the following result:

\begin{thm}\cite{forum}
     Let $h$ be a meromorphic function on $\mathbb{H}$, then $\{h,\tau\}$ is a weight $4$ automorphic form for $\Gamma$ if and only if $h$ is $\rho-$equivariant for some $2-$dimensional representation $\rho$ of $\Gamma$.
\end{thm}

It is  very useful to indicate that if $h$ is $\rho-$equivariant for some $2-$dimensional representation $\rho$ of $\Gamma$, then
$ \{ h, \tau \}$ is holomorphic on $\fH$ if and only if $h'$ does not vanish on $\fH$, that is, $h$ is everywhere locally univalent.

    It is also important to note that a solution $h$ does not uniquely determine the representation $\rho$. Indeed, if $h$ is $\rho-$equivariant, it is also $\rho'$-equivariant where $\rho'$ is any  lift of $\overline{\rho}$ to a representation of $\Gamma$. This will play an important role later on because the group structure of $\overline{\Gamma}_0(2)$ is simpler than that of $\Gamma_0(2)$.

    \begin{lem}
    Let $h$ (resp. $g$) be a nonconstant meromorphic function that is $\rho$-equivariant (resp. $\mu-$equivariant) on $\Gamma$. If
    $ \{ h, \tau\} = \{ g,\tau\}$ then  
    $$ \overline{\rho} = \sigma \overline{\mu} \sigma^{-1} \text{ for some } \sigma \in \PG.$$
\end{lem}

\begin{proof}
    By proposition 2.2, $ \{ h, \tau\} = \{ g,\tau\} $ implies that $h = \sigma \cdot g$ for some $\sigma \in \GC$. Therefore, for all $\alpha \in \Gamma$ and $\tau \in \mathbb{H}$, we have
    $$ \rho(\alpha) h(\tau) = h(\alpha \tau) =\sigma \mu(\alpha)g(\tau) = \sigma \mu(\alpha) \sigma^{-1} h(\tau),$$
    from which the result follows.
\end{proof}

The next proposition will allow us to extend our results about $\Gamma_0(2)$ to its conjugates in $\Gamma(1)$.

\begin{prop}\label{conj}
    Let $F$ be a weight $4$ modular form on $\Gamma$ and let $\lambda \in \Gamma(1)$. Then $j(\lambda,\tau)^{-4}F(\lambda \tau)$ is a weight $4$ modular form on $\lambda^{-1} \Gamma \lambda$ and
    $$ \{ h,\tau\} = F \; \iff \; \{h \circ \lambda, \tau\} = j(\lambda,\tau)^{-4}F(\lambda \tau).$$
\end{prop}

\begin{proof}
The fact that \( j(\lambda, \tau)^{-4} F(\lambda \tau) \) is an automorphic form follows from the cocycle property:
\[
j(\sigma \gamma, \tau) = j(\sigma, \gamma \tau) j(\gamma, \tau) \quad \text{for all } \sigma, \gamma \in \Gamma(1) \text{ and } \tau \in \mathbb{H}.
\]
The holomorphy of this form on \(\fH\) and at the cusps is immediate. The second part of the theorem follows from \propref{sch}.
\end{proof} 
In fact, the two conjugates of $\G_0(2)$ we will need are:
$$ \Gamma^0(2) = S \Gamma_0(2) S^{-1} = \left\{ \begin{pmatrix}
a & b \\
c & d 
\end{pmatrix} \in \Gamma(1) \; \mid \; b \equiv 0 \; \text{ mod }m \right\}$$
and
\begin{align*}
    \Gamma_\theta(2) &= V \Gamma_0(2) V^{-1}  \\
    &= \left\{ \begin{pmatrix}
a & b \\
c & d 
\end{pmatrix} \in \Gamma(1)  \mid  b,c \text{ both even or } a,d \text{ both even} \right\}
\end{align*}
where $V:=(ST)^2$. Thus, by \propref{conj}, classifying all the modular solutions for \eqref{eqn3}  will amount to classifying them for weight 4 modular forms over $\Gamma^0(2)$ and $\Gamma_\theta(2)$.

\smallskip 
We conclude this section with the following lemma, which will be useful in subsequent sections.

\begin{lem}\cite{forum}\label{forum2}
    Let $F$ be a weight $4$ modular form for $\Gamma$ and let $h$ be a meromorphic solution to $\{h,\tau\}=F$. If $\rho$ is a $2-$dimensional representation corresponding to $h$, then $\text{ker}(\overline{\rho})$ is a torsion-free normal subgroup of $\overline{\Gamma}$.
\end{lem}

\section{The $q$-expansion of the solutions at the cusps}
In this section, we fix  a level $m$ congruence subgroup $\G$ of $\text{SL}_2(\mathbb{Z})$, and let $F$ be a modular form of weight 4 for $\Gamma$. We assume that $F$ does not vanish at any cusp; a condition that will later be shown necessary for equation (1.5) to admit modular solutions. We begin by relating solutions to the equation
\begin{equation}\label{4.1}
    \{h,z\} = F
\end{equation}
 to solutions of the differential equation
\begin{equation}\label{4.2}
     y'' + \frac{F}{2}y = 0.
\end{equation}

\begin{thm}\textup{\cite{mathann}}\label{thm4.1}
    Suppose that $F$ is a weight $4$ modular form for $\Gamma$ then
    \begin{enumerate}
        \item If $y_1$ and $y_2$ are two holomorphic solutions to (4.2) then $h:=y_1/y_2$ solves \eqref{4.1}.  
        \item If $h$ is a solution to (4.1) then $h/\sqrt{h'}$ and $1/\sqrt{h'}$ define  two linearly independent holomorphic solutions to \eqref{4.2}.
    \end{enumerate}
\end{thm}

We apply the Frobenius method to the equation \eqref{4.2}  after a change of variable.

Let $q = e^{2 \pi i \tau / m}$. We look at the $q-$expansion of $F$ at the infinity cusp which we can write as
$$ F(q) = \sum_{n=0}^\infty a_nq^n,$$
where $a_n \in \mathbb{C}$ for all $n \in \mathbb{N}$ and $a_0 \ne 0$. Under this change of variable, the ODE \eqref{4.1} becomes
$$ \frac{d^2y}{dq^2} + \frac{1}{q}\frac{dy}{dq} - \frac{m^2}{8 \pi^2} 
\frac{F}{q^2} y = 0$$
defined over the punctured disc $\{q\in \mathbb{C} \mid 0 < |q| < 1\}.$
Since $a_0\neq 0$, we write  $a_0 = 2 \pi^2 \left(\frac{r}{m}\right)^2$ for some $r\in \mathbb{C}^*$ with positive real part. The above ODE  becomes
$$ \frac{d^2y}{dq^2} + \frac{1}{q}\frac{dy}{dq} - \frac{r^2}{4} 
\frac{\Tilde{F}}{q^2} y = 0, $$
which is a
Fuchsian differential equation with a regular singular point at $q=0$. The corresponding indicial equation is $x^2-\frac{r^2}{4} = 0$ which has roots $\pm r/2$. A first solution to the ODE is given by
$$
y_1 = q^{r/2}\sum_{n=0}^\infty \alpha_n q^n\, ,\ \alpha_n \in \mathbb{C}\, ,\ \alpha_0 \ne 0.
$$
The second (linearly independent with $y_1$) solution depends on whether $r$ is an integer or not.

If $r \notin \mathbb{Z}$, then we can take 
$$ y_2 = q^{-r/2}\sum_{n=0}^\infty \beta_n q^n \,,\ \beta_n \in \mathbb{C}\, ,\ \beta_0 \ne 0,$$

while if $r \in \mathbb{Z}$, then the second solution is not meromorphic at $q=0$ but it has a logarithmic singularity and is given by
$$ y_2 = c \, \text{log}(q) \, y_1(q) + q^{-r/2} \sum_{n=0}^\infty \beta'_n q^n,$$
where $c , \beta'_n \in \mathbb{C}$ and $c,\beta'_0 \ne 0$.

Having these linearly independent solutions, we can use \thmref{thm4.1} to deduce a general solution to \eqref{4.2}.

\begin{thm}\label{thm4.2}
    Let \( F = \sum_{n=0}^\infty a_n q^n \) be a weight 4 modular form such that \( a_0 = 2\pi^2 \left( \frac{r}{m} \right)^2 \), where \( r \in \mathbb{C}^* \) has positive real part. Consider the equation
    \[
    \{h, \tau\} = F.
    \]
    Then the following holds:

    \begin{enumerate}
        \item If \( r \notin \mathbb{Z} \), there exists a meromorphic solution \( h(\tau) \) of the form
        \[
        h(\tau) = q^r \sum_{n=0}^\infty a_n q^n \quad \text{with } a_0 \neq 0.
        \]
        \item If \( r \in \mathbb{Z} \), there exists a meromorphic solution \( h(\tau) \) of the form
        \[
        h(\tau) = \tau + q^{-r} \sum_{n=0}^\infty b_n q^n \quad \text{with } b_0 \neq 0. \]
        \end{enumerate}
            
    Furthermore, any other solution to the equation \( \{h, \tau\} = F \) can be expressed as a linear fractional transformation of \( h \).
    \end{thm}

In what follows, we will show that, in seeking a modular solution to \eqref{4.2}, the corresponding modular form $F$ must be non-vanishing at all cusps.
\begin{prop}\label{prop4.3}
    If $h$ is a modular function for a finite index subgroup $\G$ of $\SL$, then $F(\tau)=\{h,\tau\}$ does not vanish at any cusp of $\G$.
\end{prop}
\begin{proof}
    Let $h$ be a modular function for a finite index subgroup of $\G$, and let $m$ be the smallest positive integer such that $T^m\in\G$. Then $h$ admits a meromorphic $q-$expansion at $\infty$, where $q=\exp(2\pi i\tau/m)$. 
Since the Schwarz derivative is projectively invariant, we can assume that the $q-$expansion of $h$ has the form 
$$ h(\tau)=\sum_{n\geq n_0}\,a_n\,q^n,$$ 
where $n_0>0$.     
    Using \eqref{schwarz}, we find 
    $$\disp F(\tau)=\{h,\tau\}=2{\pi}^2 (n_0/m)^2 +\mbox{O}(q),$$
showing that $F$  is holomorphic and nonvanishing at $\infty$. This conclusion can also be drawn from the following reasoning: if $F$ were to vanish at $\infty$, then by applying the Frobenius method to the differential equation 
$\{h,\tau\}=F(\tau)$, we obtain an indicial equation given by $x^2 = 0$ which has a double root $0$. According to \thmref{thm4.2}, this would imply the existence of a solution $h_0$ with a logarithmic singularity at $\infty$ contradicting the fact that $h_0$ is a linear fractional transformation  of $h$. Consequently, $F$ cannot vanish at $\infty$. Next, suppose that  $F$ vanishes at a cusp $s$. Choose $\gamma\in\SL$ such that $\gamma\cdot s=\infty$. Then $g= f \circ \gamma$ is a modular function for the group $\gamma^{-1}\G\gamma$ and satisfies the equation $\{g , \tau \} = j(\gamma, \tau)^{-4} F(\gamma \tau)$, which vanishes at $\infty$, leading to a contradiction with the earlier conclusion.
\end{proof}

As an example, the equation 
$$ \{ h, \tau \} = c \, \theta_2^8$$
has no modular solutions for all values of $c \in \mathbb{C}$.
\smallskip

 For the remainder of this section, let $F$  denote a modular form of weight $4$ over $\Gamma_0(2)$ that does not vanish at the cusps, and $h$  be a meromorphic solution to $\{h,\tau\}=F$. From now on, let $q = e^{2 \pi i \tau}$ and $\rho$ represent the 2-dimensional complex representation induced by $h$ and by $\bar{\rho}$ its projectivization. 

\begin{thm}\label{thm4.4}
    If $\overline{\rho}$ has a finite image, then $h$ is a modular function on $\ker(\overline{\rho})$ with following $q$-expansions at $\infty$, up to fractional linear transformation,
    \begin{enumerate}
        \item At the cusp $\infty$:
        \[
        h(\tau)\,=\,q^{\frac{n_1}{m_1}} \sum_{n=0}^\infty a_nq^n\,\ \ a_i\in\BC\,\ a_0\neq 0,
        \]
        where $m_1$ is the cusp width at $\infty$, that is, the smallest positive integer $m_1$ such that $T^{m_1}\in \text{ker}(\overline{\rho})$, and $n_1$ is a positive integer with $\gcd(m_1,n_1)=1$.
        \item A the cusp 0: 
        \[
        h(-1/\tau)\,=\,q^{\frac{n_2}{m_2}} \sum_{n=0}^\infty b_nq^n\,\ \ a_i\in\BC\,\ b_0\neq 0,
        \]
        where $m_2$ is the cusp width at $0$, that is, the smallest positive integer $m_2$ such that $ST^{m_2}S^{-1}\in \text{ker}(\overline{\rho})$, and $n_2$ is a positive integer with $\gcd(m_2,n_2)=1$.
        \end{enumerate}
\end{thm}

\begin{proof}
It follows  immediately from the definition of a $\rho-$equivariant function that $h(\sigma \tau) = h(\tau) $
for all $\sigma \in \ker(\overline{\rho})$. We only  need to show  that it is meromorphic at the cusps. By definition of $m_1$, we have $h(\tau+m_1) = h(\tau)$ and this clearly holds for any linear fractional transformation of $h$. Therefore, by theorem 4.2, $h$  cannot have a logarithmic pole at $\infty$ and must take the form
$$ h(\tau) = q^r\sum_{n=0}^\infty a_n q^n,$$
for some $a_i,r \in \mathbb{C}$ such that $a_0 \ne 0$ and $\text{Re}(r) \geq 0$. By the periodicity of $h$, we conclude that $n_1 = m_1 r \in \mathbb{Z}_{\geq 1}$ and hence $r = n_1/m_1$ proving that $h$ is meromorphic at $\infty$.

As for the cusp at $0$, recall that $S\cdot 0 = \infty$. For $g := h \circ S$, we have 
$$ \{g , \tau \} = j(\gamma, \tau)^{-4} F(\gamma \tau). $$
By \propref{conj}, the right hand side of the equation is also a modular form of weight $4$ on $\Gamma^0(2)$. The remainder of the proof follows similarly to the case of the cusp at $\infty$.

\end{proof}

Conversely, if the $\rho-$equivariant function $h$ is a modular function for a finite index subgroup $\Gamma \leq \Gamma(1)$, then $\overline{\G}\leq \ker{\bar{\rho}}$, and therefore $\bar{\rho}$ has a finite image. 
%

Recall that the space $M_4(\Gamma_0(2))$ of weight $4$ modular forms on $\Gamma_0(2)$ is a $2$-dimensional vector space over $\mathbb{C}$, for which choose the following basis :
$$ M_4(\Gamma_0(2)) = \mathbb{C} \, \theta_2^8 \oplus \mathbb{C} \left( \theta_3 \theta_4 \right)^4.$$
For completeness, we also include the bases for the related spaces:
$$ M_4(\Gamma^0(2)) = \mathbb{C} \, \theta_4^8 \oplus \mathbb{C} \left( \theta_2 \theta_3 \right)^4 \; \text { and } 
\; M_4(\Gamma_\theta(2)) = \mathbb{C} \, \theta_3^8 \oplus \mathbb{C} \left( \theta_2 \theta_4 \right)^4.$$

Equation \eqref{4.1} can now be expressed as
\begin{equation}\label{4.3}
    \{h,\tau\} =  a \left( \theta_3 \theta_4 \right)^4 + b \left(\theta_2\right)^8 ,
\end{equation}
where $a,b \in \mathbb{C}$.
\begin{thm}\label{thm5.1}
If the Schwarzian equation \eqref{4.3} admits a modular function as a solution, then 
    $$
    b = 2 \pi^2 \left(\frac{n_1}{m_1}\right)^2 \; \; \; \text{ and } \; \; \; a = 2 \pi^2 \left(\frac{n_2}{m_2}\right)^2,
    $$
    where $n_1$, $n_2$, $m_1$ and $m_2$ are positive integers  defined as in \thmref{thm4.4}. 
\end{thm}

\begin{proof}
    First, we note that \propref{prop4.3} implies that $a$ and $b$ are both non-zero, as $\{h,\tau\}$ takes the value $a$ at the cusp $\infty$ and $b$  at the cusp  $0$. 
    By \thmref{thm4.2}, there exists relatively prime positive integers $m_1,n_1$  such that $h$ is a linear fractional transformation of
    \[
        q^{\frac{n_1}{m_1}}\sum_{i=0}^\infty a_i q^i.
    \]
    Using the cocycle relation \eqref{cocycle} with $w=q$ and noting that  $\{q, \tau \} = 4 \pi^2$, we obtain
\[
        \{ h, \tau\} = 2 \pi^2 \left( \frac{n_1}{m_1}\right)^2 + \mbox{O}\left(q^{\frac{n_1}{m_1}}\right). 
\]
    Now, since  $ \{h,\tau\} =  a \left( \theta_3 \theta_4 \right)^4 + b \left(\theta_2\right)^8 = a +  \mbox{O}(q)$, we conclude that 
    $$
    a= 2 \pi^2 \left( \frac{n_1}{m_1}\right)^2.
    $$ 
    Moreover, $h$ is also meromorphic at $0$ and therefore $h \circ S $ takes the form of a linear fractional transformation of
  \[
         q^{\frac{n_2}{m_2}}\sum_{i=0}^\infty b_i q^i,
    \]
    whose Schwarz derivative has the form 
\[
        \{ h\circ S, \tau\} = 2 \pi^2 \left( \frac{n_2}{m_2}\right)^2 + \mbox{O}\left(q^{\frac{n_2}{m_2}}\right). 
\]
  In the meantime, using the transformation rules for the theta functions from Section 2,  we find
    \begin{align*}
        \{h \circ S, \tau\} &= \tau^{-4} \left[ a \left( \theta_3 \theta_4 \right )^4 (S \tau) + b  \left(\theta_2\right)^8   (S \tau)\right]  \\
        &=    a \left( \theta_3 \theta_2 \right )^4  + b  \left(\theta_4\right)^8 \\
        &= b + \mbox{O}(q).
    \end{align*}
    Thus we conclude that 
    $$
    b = 2 \pi^2 \left( \frac{n_2}{m_2}\right)^2 .
    $$
\end{proof}


\section{Finite image Representations  of $\Gamma_0(2)$}

We will study the modular solutions to the equation \eqref{4.3} by classifying them according to their respective finite image projective representations. Finite subgroups of $\GC$ are well known and they are the binary polyhedral groups. Their projection into $\PG$ is easier to work with.

\begin{prop}\cite{klein}\label{klein}
    The finite subgroups of $\PG$ are $A_5$, $S_4$, $A_4$, $D_{2n}$ for $n \geq 2$ and $C_n$ for $n \geq 1$. Moreover, for each of these groups, there is only one conjugacy class.
\end{prop}

Thus, if $h$ is a modular solution to \eqref{4.3}, the image of its associated projective representation must be isomorphic to one of the groups listed in the above proposition. However, not all of these groups  necessarily correspond to solutions of \eqref{4.3}. Indeed, when $n$ is odd, suppose there exists a solution $h$ whose corresponding projective representation is $\overline{\rho}: \overline{\Gamma}_0(2) \to C_n$. Since $RT^{-1}$ has order $2$ in $\overline{\Gamma}_0(2)$, its image must have order $1$, implying that $RT^{-1} \in \ker{\overline{\rho}}$. This, however, is impossible because the kernel must be torsion-free by \lemref{forum2}. For the remaining groups, we will compute all the possible kernels of their corresponding projective representation.

To proceed, we introduce some notation. Let $G$ be a group and $H$ a subgroup. Denote by $G'$ the commutator subgroup of $G$ and by $N_G(H)$ the normal closure of $H$ in $G$, i.e., the smallest normal subgroup of $G$ containing $N$.

\begin{thm}\label{groups}
    Let $\rho : \overline{\Gamma}_0(2) \to \PG$ be a finite image projective representation with torsion-free kernel. 
   Then its image $G$ and its kernel  $\G$ fall into one of the following cases :

    \begin{enumerate}
        \item $G=A_4$ and $\Gamma =  N_{\overline{\Gamma}_0(2)}(T^3, R^3) = \Gamma_0(2) \cap \Gamma(3)$.
         \item $G=S_4$ and either $\Gamma= N_{\overline{\Gamma}_0(2)}(T^4, R^3)$ or $\Gamma = N_{\overline{\Gamma}_0(2)}(T^3, R^4)$.
        \item $G=A_5$ and either $\Gamma = N_{\overline{\Gamma}_0(2)}(T^5, R^3)$ or $\Gamma = N_{\overline{\Gamma}_0(2)}(T^3, R^5)$.
    
        \item $G=D_{2n}$ and either $\Gamma = N_{\overline{\Gamma}_0(2)}(T^n, R^2)$ or  $\Gamma = N_{\overline{\Gamma}_0(2)}(T^2, R^n)$.
        \item $G=C_{2n}$ for $n$ even integer and 
        $$\Gamma = \left\langle T^{2n}, R^{2n}, T^{n+1}R^{-1}, \Gamma_0(2)' \right\rangle.$$
        \item $G=C_{2n}$ for $n$ odd integer and either
        $$\Gamma = \left\langle T^{n}, R^{2n}, R^{n+1}T^{-1}, \Gamma_0(2)' \right\rangle \text{ or } \Gamma = \left\langle T^{2n}, R^{n}, T^{n+1}R^{-1}, \Gamma_0(2)' \right\rangle.$$
    \end{enumerate}

\end{thm}

\smallskip

\begin{proof}
(1) Let $\rho : \overline{\Gamma}_0(2) \twoheadrightarrow A_4$ be a surjective group homomorphism. Observe that $\rho(T)$ and $\rho(RT^{-1})$ generate $A_4$. Since $RT^{-1}$ is an elliptic element with order $2$,  $\rho(RT^{-1})\neq  1$ because $\text{ker}(\rho)$ is torsion-free. Thus,  $\rho(RT^{-1})$ has order $2$ and represents the product of two disjoint 2-cycles. Consequently, $\rho(T)$ must be a 3-cycle, as $A_4$ cannot be generated by two elements of order $2$. Moreover, in $A_4$, the product of a 3-cycle with the product of two disjoint two cycles is itself a 3-cycle, implying that $\rho(R)$ has order $3$ in $A_4$. Thus, we have
 $$ N_{\overline{\Gamma}_0(2)} (T^3,R^3) \subseteq \text{ker}(\rho).$$
To establish that this inclusion is actually an equality, it suffices to show that 
$$ \overline{\Gamma}_0(2) / N_{\overline{\Gamma}_0(2)} (T^3,R^3) \cong A_4. $$
This follows directly, as   under the isomorphism \eqref{isom}, we find
$$
\overline{\Gamma}_0(2) / N_{\overline{\Gamma}_0(2)} (T^3,R^3) 
\, \cong \, \left< a,b \mid b^2=1 \right> / N(a^3, (ba)^3),
$$
and $ A_4$ admits the presentation $\left<a,b \mid a^3=b^2=(ba)^3=1\right>$.

Additionally, recall that $\overline{\Gamma}(3)$ is a normal subgroup of $\overline{\Gamma}(1)$, and so $\overline{\Gamma}(3) \cap \overline{\Gamma}_0(2)$ is a normal subgroup of $\overline{\Gamma}_0(2)$ that clearly  contains $T^3$ and $R^3$. Therefore, we conclude $\text{ker}(\rho) \subseteq \overline{\Gamma}(3) \cap \overline{\Gamma}_0(2)$. The desired equality follows from the fact that 
$$ \overline{\Gamma}_0(2)/ \overline{\Gamma}(3) \cap \overline{\Gamma}_0(2) \cong \overline{\Gamma}_0(2) \overline{\Gamma}(3)/ \overline{\Gamma}(3) \cong \overline{\Gamma}(1)/\overline{\Gamma}(3) \cong A_4$$
where the last isomorphism can be found in \cite{mason2}.

(2) To ensure that a group homomorphism  $\rho : \overline{\Gamma}_0(2) \twoheadrightarrow S_4$ is surjective, it is necessary for $\rho(T)$ to have order $3$ or $4$. We will examine each case separately. Note, however, that in both cases, $\rho(RT^{-1})$ must have order $2$.

If $\rho(T)$ has order $3$, then  $\rho(R) = \rho(RT^{-1}) \rho(T)$ has order $4$. This is because permutations of order $3$ in $S_4$ are 3-cycles, which are even permutations. Therefore,  $\rho(RT^{-1})$ must be an odd permutation of order $2$, i.e., a transposition, as the two generators cannot both have an even signature. On the other hand, if a 2-cycle and a 3-cycle generate $S_4$,  their product is always a 4-cycle. Thus,  $\rho(R)$ has order $4$ leading to the inclusion
$$
N_{\overline{\Gamma}_0(2)}(T^3,R^4) \subseteq \text{ker}(\rho). 
$$ 
To establish the equality, we use the same reasoning as in the $A_4$ case using the fact that 
$S_4$ has the presentation 
$$
\left<a,b \mid a^4 = b^2 = (ba)^3 =1\right>.
$$
The case where $\rho(T)$ has order $4$ follows similarly.

(3) and (4): The $A_5$  and  $D_{2n}$ cases can be shown using similar arguments as above.

(5) Now suppose that $\rho : \overline{\Gamma}_0(2) \twoheadrightarrow C_{2n}$ is surjective, where $n$ is even. Since $C_{2n}$ is abelian, the representation $\overline{\rho}$ factors through the commutator subgroup $\overline{\Gamma}_0(2)'$. We look at $C_{2n}$ as an additive group, with  the order 2 element $\rho(RT^{-1})$ corresponding to the residue class of $n$ in $C_{2n}$. Now, let $a = \rho(T)$. For $\rho$ to be surjective, we must have $\text{gcd}(a,n) = 1$. Given that $n$ is even, it follows that  $\text{gcd}(a,2n) = 1$ as well. Thus,  $a$ has order $2n$ in $C_{2n}$ and $\rho(T)^n=\rho(RT^{-1})$. We conclude that 
$$ \Gamma := \left<\overline{\Gamma}_0(2)', T^{2n}, T^{n+1}R^{-1}\right> \, \subseteq \,\ker{\rho}.$$
 All  elements of $\overline{\Gamma}_0(2)/\Gamma$ can be expressed powers of $T$, meaning the quotient has at most $2n$ elements. The above inclusion is therefore an equality.

(6) Let $\rho: \overline{\Gamma}_0(2) \to C_{2n}$ be a surjective homomorphism with   odd $n$. Set $a := \rho(T)$. Since $\rho$ is surjective, we must have $\text{gcd}(a,n)=1$; otherwise, the order of $a$ would be less than $n$, preventing $a$ and $\rho(RT^{-1})$ from generating $C_{2n}$, which has an element of order $2n$. 

There are two cases to consider. If $a$ is odd, then $\text{gcd}(a,2n) = 1$ and $\rho(T)$ has order $2n$. We also have $\rho(R) =\rho(RT^{-1})+\rho(T)= n + a$ and one can verify that $\text{gcd}(a+n,2n) = 2$, implying that $\rho(R)$ has order $n$. This leads us to the inclusion 
$$ \Gamma := \left<\overline{\Gamma}_0(2)', T^{2n},R^n, T^{n+1}R^{-1}\right> \, \subseteq \,\ker{\overline{\rho}}.$$   We show that the above inclusion is an equality using the same reasoning as in the case of $n$ even.

If $a$ is even, then a similar argument shows that
$$ \ker{\overline{\rho}} =  \left<\overline{\Gamma}_0(2)', T^{2n},R^n, T^{n+1}R^{-1}\right>.$$
  
\end{proof}

\begin{remark}{\em
The reason that $\ker{\overline{\rho}}$ has two possible forms in some cases is as follows:
The Fricke involution $\begin{pmatrix}
0 & 1/\sqrt{2} \\
-\sqrt{2} & 0 
\end{pmatrix} \in \SR$ normalizes $\G_0(2)$ and it is transitive on the cusps and at the same time conjugates $T$ to $R^{-1}$. If  $\ker{\overline{\rho}}$ has only one possible form, it would be fixed by the action of this involution.
}
\end{remark}

For the remainder of this section, we will study each subgroup $\G$ of $\overline{\Gamma}_0(2)$ that appears in \thmref{groups}  in more details. In particular, we will compute the genus of their corresponding modular curve $X(\G)=\G\backslash (\fH\cup \mathbb P_1(\BQ))$.

Now with ${\Gamma}$ being a finite index subgroup of ${\Gamma}(1)$, the  inclusion ${\Gamma} \subseteq {\Gamma}(1) $ induces a natural finite covering of Riemann surfaces 
$$ X\left(\Gamma \right) \to X\left(\Gamma(1) \right).$$
Using  the Riemann-Hurwitz formula for this covering, the genus of $X\left(\Gamma \right)$ is given by
$$ g = 1 + \frac{\mu}{12} - \frac{e_2}{4} - \frac{e_3}{3} - \frac{v}{2},$$
where $\mu = \left[ \overline{\Gamma}(1) : \overline{\Gamma}\right]$, $e_k$, $k=1$ or $2$, is the number of $\overline{\Gamma}-$inequivalent elliptic fixed points of order $k$, and $v$ is the number of  $\Gamma-$inequivalent cusps. Since ${\Gamma}$ is torsion-free, the formula simplifies to
\begin{equation}\label{genus}
g = 1 + \frac{\mu}{12} - \frac{1}{2}v.
\end{equation}
In order to determine $v$, we use the orbit-stabilizer theorem for infinite groups.
\begin{prop}\cite{shimura}
     Let $G$ be a group acting on a set $S$ transitively and let $\G$ be a finite index subgroup of $G$. Then, for all $x \in S,$ the stabilizer of $x$ in $\G$ has a finite index in the stabilizer of $x$ in $G$. Moreover,
    $$ [G:\G] = \sum_{x \in \G \backslash S} \left[ \text{Stab}_G(x): \text{Stab}_{\G}(x) \right].$$
\end{prop}

Denote by $[0]$ (resp. $[\infty]$) the set of $\Gamma-$inequivalent cusps of $\Gamma$ that are $\overline{\Gamma}_0(2)$-equivalent to $0$ (resp. $\infty$) and by $v_0$ (resp. $v_{\infty}$) the cardinality of $[0]$ (resp. $[\infty]$). Note that we have $v= v_0 + v_\infty$. Additionally, $\Gamma = \ker{\overline{\rho}}$, being a normal subgroup of $\overline{\Gamma}_0(2)$,  acts transitively on each of the sets $[0]$ and $[\infty]$.

 Let $s \in [\infty]$. There exists $\gamma \in \overline{\Gamma}_0(2)$ such that $\gamma \infty = s$. Thus,
$$ \text{Stab}_{\overline{\Gamma}_0(2)} (s) = \gamma \text{Stab}_{\overline{\Gamma}_0(2)} (\infty) \gamma^{-1} = \gamma \left<T\right> \gamma^{-1}.$$

Similarly
$$ 
\text{Stab}_{\Gamma} (s) = \gamma \text{Stab}_{\gamma^{-1}\Gamma \gamma} (\infty) \gamma^{-1} = \gamma \text{Stab}_{\Gamma} (\infty) \gamma^{-1} = \gamma \left<T^{m_1}\right> \gamma^{-1},
$$
where the cusp width $m_1$ is as in \thmref{thm4.4}. Therefore, using the orbit-stabilizer theorem, we have
$$ 
[\overline{\Gamma}_0(2):\Gamma] = \sum_{s \in [\infty]} m_1 = v_\infty m_1.
$$
Similarly, if $s$ is a cusp of $\Gamma$ in $[0]$ and $\gamma s = 0$ for some $\gamma \in \overline{\Gamma}_0(2)$, then
$$ \text{Stab}_{\overline{\Gamma}_0(2)} [0] = S \,\text{Stab}_{S^{-1}\overline{\Gamma}_0(2) S}(\infty) \, S^{-1} = S\left<T^2\right>S^{-1}.$$
We also have  $\text{Stab}_{\Gamma} (0) = S \left<T^{m_2}\right> S^{-1}$, where $m_2$ is the cusp width of $0$ in $\G$. Hence,
\begin{align*}
    [\overline{\Gamma}_0(2):\Gamma] &= \sum_{s \in [0]} [\text{Stab}_{\overline{\Gamma}_0(2)}(s) : \text{Stab}_{\Gamma}(s) ] \\
    &= \sum_{s \in [0]} [\gamma \text{Stab}_{\overline{\Gamma}_0(2)}(0) \gamma^{-1}  : \gamma \text{Stab}_{\Gamma}(0) \gamma^{-1} ] \\
    &= \sum_{s \in [0]} [(\gamma S)\text{Stab}_{S^{-1}\overline{\Gamma}_0(2)S}( \infty) (\gamma S)^{-1}  : (\gamma S) \text{Stab}_{S^{-1} \Gamma S}(\infty) (\gamma S)^{-1} ] \\
    &= \sum_{s \in [0]} [(\gamma S)\left<T^2\right>(\gamma S)^{-1}  : (\gamma S) \left<T^{m_2}\right> (\gamma S)^{-1} ] \\
    &= v_0 \frac{m_2}{2}.
\end{align*}

If we set $\mu' := [\overline{\Gamma}_0(2):\Gamma]$, then $\mu = 3 \mu'$ and 
$$
v = v_0 + v_\infty = \mu' \left( \frac{1}{m_1} + \frac{2}{m_2} \right).
$$
Using these expressions and the Riemann Hurwitz formula \eqref{genus}, we get

\begin{prop}
Let $h$ be a modular solution to equation 5.3 and let $\overline{\rho} : \overline{\Gamma}_0(2) \twoheadrightarrow G$ be its corresponding projective representation, then $h$ is invariant under the modular subgroup $\ker{\overline{\rho}}$ whose genus is given by
    $$ g= 1 + |G| \left(\frac{1}{4} - \frac{1}{2m_1} - \frac{1}{m_2}\right),$$
where $m_1$ and $m_2$ are the respective cusp widths at $\infty$ and at $0$.
\end{prop}

This section can now be summarized in the following table:

\begin{table}[ht]
\centering
\begin{tabular}{>{\centering\arraybackslash}m{3cm} >{\centering\arraybackslash}m{5cm} >{\centering\arraybackslash}m{2cm} >{\centering\arraybackslash}m{1.5cm}}
\toprule
\textbf{G} & \boldmath{$\Gamma$} & \textbf{(m\textsubscript{1}, m\textsubscript{2})} & \textbf{Genus} \\
\midrule
$A_4$ & $\Gamma_0(2) \cap \Gamma(3)$ & $(3,6)$ & $0$ \\
\midrule
\multirow{2}{*}{\rule{0pt}{2.5ex} $S_4$ \rule{0pt}{2.5ex}} & $N_{\Gamma_0(2)}(T^4,R^3)$ & $(4,6)$ & \multirow{2}{*}{$0$} \\
& $N_{\Gamma_0(2)}(T^3,R^4)$ & $(3,8)$ & \\
\midrule
\multirow{2}{*}{\rule{0pt}{2.5ex} $A_5$ \rule{0pt}{2.5ex}} & $N_{\Gamma_0(2)}(T^5,R^3)$ & $(5,6)$ & \multirow{2}{*}{$0$} \\
& $N_{\Gamma_0(2)}(T^3,R^5)$ & $(3,10)$ & \\
\midrule
\multirow{2}{*}{\rule{0pt}{2.5ex} $D_{2n}$ for $n \geq 1$ \rule{0pt}{2.5ex}} & $N_{\Gamma_0(2)}(T^n,R^2)$ & $(n,4)$ & \multirow{2}{*}{$0$} \\
& $N_{\Gamma_0(2)}(T^2,R^n)$ & $(2,2n)$ & \\
\midrule
$C_{2n}$ for $n$ even & $\left< T^{2n}, R^{2n}, T^{n+1}R^{-1}, [T, R]\right> $ & $(2n,4n)$ & $\frac{n}{2}$ \\
\midrule
\multirow{2}{*}{\rule{0pt}{2.5ex} $C_{2n}$ for $n$ odd \rule{0pt}{2.5ex}} &  $\left< T^{n}, R^{2n}, R^{n+1}T^{-1}, [T, R]\right> $& $(n,4n)$ & \multirow{2}{*}{$\frac{n-1}{2}$} \\
&  $\left< T^{2n}, R^n, T^{n+1}R^{-1}, [T, R]\right> $ & $(2n,2n)$ & \\
\bottomrule
\end{tabular}
\label{tab:my_label}
\end{table}

\bigskip

\begin{thm}\label{thm5.6}
A solution to the differential equation
\[
\left\{ h, z \right\} = a \, (\theta_2^8) + b \, (\theta_3 \theta_4)^4
\]
is a modular function for a genus $0$ 
 subgroup of the modular group if and only if
\[
a = 2 \pi^2 \left(\frac{n_1}{m_1}\right)^2 \quad \text{and} \quad b = 2 \pi^2 \left(\frac{n_2}{m_2}\right)^2,
\]
where $n_1, m_1$ (resp. $n_2, m_2$) are relatively prime positive integers, and one of the following holds:
\begin{enumerate}
    \item $a = b$ and $m_1 = m_2 \in \{2, 3, 4, 5\}$, \\
    \item $(m_1, m_2) \in \{(n, 4) \mid n \geq 1\} \cup \{(2, 2n) \mid n \geq 1\}$, \\
    \item $(m_1, m_2) \in \{(3, 6), (4, 6), (3, 8), (5, 6), (3, 10)\}$.\\
\end{enumerate}
\end{thm}

\smallskip

\begin{proof}
First, note that if $a=b$,  the invariance subgroup of $h$ in $\overline{\Gamma} (1)$ may be larger than $\ker(\overline{\rho})$. Indeed, in this case,  $h$ becomes $\rho-$equivariant for a representation $\rho$ of $\Gamma(1)$. The case $a=b$  has already been  treated in \thmref{forum1}. This, together with the above table prove  the ``if" part of the theorem.

For the converse, suppose $a \ne b$. Solving
    \begin{equation}\label{5.2}
\left\{ h, z \right\} = 2 \pi^2 \left(\frac{n_1}{m_1}\right)^2 \, (\theta_2^8) + 2 \pi^2 \left(\frac{n_2}{m_2}\right)^2 \, (\theta_3 \theta_4)^4
\end{equation}
with the Frobenious method implies that $h$ is a linear fractional transformation of an expression of the form 
$$ q^{\frac{n_1}{m_1}}\sum_{n=0}^\infty a_i q^i.$$
In particular, we observe that $T^{m_1} \in \ker{\overline{\rho}}$. Similarly, by considering $h \circ S$, we deduce that 
$$
R^{\frac{m_2}{2}} = ST^{-m_2}S^{-1}  \in \ker{\overline{\rho}}.
$$
This implies that 
$$ 
N_{\Gamma_0(2)}(T^{m_1},R^{\frac{m_2}{2}}) \subseteq \ker{\overline{\rho}},
$$
and for all the above cited  values of $m_1$ and $m_2$, the group $N_{\Gamma_0(2)}(T^{m_1},R^{\frac{m_2}{2}})$ 
is a finite index genus $0$ subgroup of the modular group.
\end{proof}

Later, when we  examine the modular curve associated to reducible representations, we will remove the genus 
0 assumption imposed in the preceding theorem.

\section{Explicit Solutions: The Genus $0$ Case}

Let $h$ be a modular solution to \eqref{4.3} and let $\rho$ be its corresponding representation. Set $\Gamma = \ker{\overline{\rho}}$. The map $h$ induces a natural covering of Riemann surfaces 
$$ h : X(\Gamma) \to \mathbb{P}_1(\mathbb{C}).$$
The Riemann-Hurwitz formula for this covering is given by
$$ 2g - 2 = -2d + \sum_{P \in \mathbb{P}_1(\mathbb{C})} e_P -1,$$
where $d$ is the degree of the covering, $g$ is the genus of $\Gamma$, and $e_P$ is the ramification index above $P$.

Since $\{h,\tau\}$ is holomorphic on $\mathbb{H}$, we deduce from the definition of the Schwarz derivative, that $h$ takes its finite values only once (as $h'$ does not vanish and can only have  simple poles in $\mathbb{H}$). Consequently, $h$ can  be ramified only at the cusps. Furthermore, by choosing a linear fractional transformation of $h$ that vanishes at $\infty$ (resp. $0$), then $n_1$ (resp. $n_2$) is the order of vanishing. Therefore, every cusp of $\Gamma$ that is $\Gamma_0(2)$-equivalent to $\infty$ (resp. $0$) has ramification index $n_1$ (resp. $n_2$).

The Riemann-Hurwitz formula can then be written as :
\begin{equation}\label{6.1}
    d = 1 - g + \frac{\mu}{2 m_1}(n_1 -1) + \frac{\mu}{m_2}(n_2 -1).
\end{equation} 

\subsection*{The Tetrahedral Case}

The simplest case involves the modular solutions $h$ to \eqref{4.3} whose projective representation has an image isomorphic to $A_4$. In this case, the degree formula \eqref{6.1} simplifies to
\begin{equation}\label{degree}
d = 2 (n_1 + n_2) - 3.
\end{equation}
When $n_1 = n_2 =1$, the solution $h$ is a degree $1$ modular function, that is, a hauptmodul on the group
$$
\Gamma = \overline{\Gamma}_0(2) \cap \overline{\Gamma}(3).
$$  
One such hauptmodul is given by
$$t(\tau)= \frac{\eta(2 \tau) \eta(3 \tau)^3}{\eta(\tau) \eta(6 \tau)^3},$$
and as all hauptmoduln are linear fractions of each other, \thmref{thm5.6} implies:
$$ 
\left\{ t,\tau\right\} =  2 \pi^2 \left(\frac{1}{6}\right)^2 \, (\theta_2)^8 + 2 \pi^2 \left(\frac{1}{3}\right)^2 \, (\theta_3\theta_4)^4.
$$
For  a general equation of the form
$$ \left\{ h,\tau\right\} =  2 \pi^2 \left(\frac{n_1}{6}\right)^2 \, (\theta_2)^8 + 2 \pi^2 \left(\frac{n_2}{3}\right)^2 \, (\theta_3\theta_4)^4,$$
\thmref{thm5.6} implies that any solution will be a modular function over $\Gamma$. Therefore, $h$ must be  a rational function of the hauptmodul $t$, with a degree  given by \eqref{degree}. In what follows, we compute some examples of these rational functions. 

\bigskip

\renewcommand{\arraystretch}{2} 

\begin{center}
\begin{tabularx}{0.7\textwidth} { 
  | >{\centering\arraybackslash}X 
  | >{\centering\arraybackslash}X 
  | }
 \hline
 $(n_1, n_2, d)$ & $h$ \\
 \hline
 (1,2,3)  & $\frac{t}{t^3+t-2}$    \\
\hline
 (1,4,7)  & $\frac{t^7+14t^4-14t}{7t^3-2}$    \\
\hline
 (5,1,9)  & $\frac{t^8+16t^5}{5t^9+240t^6+384t^3-256}$    \\
\hline
 (1,5,9)  & $\frac{2t^4-t}{t^9+12 t^6-60t^3+10}$    \\
\hline
 (5,2,11)  & $\frac{55t^9+528t^6+256}{t^{11}+110 t^8 +352t^5}$    \\
\hline
\end{tabularx}
\end{center}

\subsection*{The Dihedral Case}
In this section, we focus on finding modular solutions $h$ to \eqref{4.3} where the associated projective representation has an image isomorphic to  $D_{2n}$ for some  positive integer $n$. Some of the modular functions that will appear in this section are defined on Fermat curves and were studied in detail by Rohrlich \cite{roh,roh2}.

We begin the case where
$$
\Gamma := \ker{\overline{\rho}} = N_{\overline{\Gamma}_0(2)}(T^2,R^n),
$$
and, in particular, when $n=1$, we have $\G=\overline{\Gamma}(2)$ which has a hauptmodul
$$
\lambda(\tau) = \frac{\theta_2^4}{\theta_3^4} = 16 \, q^{1/2}\prod_{k=1}^\infty \left(\frac{1+q^k}{1+q^{k-\frac{1}{2}}}\right)^8.
$$
Another hauptmodul for this group is
$$ 
1- \lambda(\tau)=\lambda(S\tau) = \frac{\theta_4^4}{\theta_3^4} = \prod_{k=1}^\infty \left(\frac{1-q^{k-\frac{1}{2}}}{1+q^{k-\frac{1}{2}}}\right)^8.
$$
The function  $1-\lambda$ is holomorphic and non-vanishing on $\mathbb{H}$ and at $\infty$. Its logarithm, defined as
$$ 
\log (1-\lambda) := 8 \sum_{n=1}^\infty  \log(1-q^{k-\frac{1}{2}}) - \log(1+q^{k-\frac{1}{2}}),
$$ 
where $\log$ denotes the principal branch of the logarithm, is holomorphic on $\mathbb{H}$ and so is $(1-\lambda)^{\frac{1}{n}} := e^{\frac{1}{n} log(1-\lambda)} $.

On can verify, using the transformation rules of the theta functions that
$$ 
(1-\lambda)^{\frac{1}{n}} \circ T = \frac{1}{(1-\lambda)^{\frac{1}{n}}}.
$$
Next, we compute $(1-\lambda)^{\frac{1}{n}} \circ R$. For this, we rely on a result due to  Rorhlich, which we now state.

\begin{thm}\cite{roh}
    Let $G$ be a finite index subgroup of $\SL$ and, for some fixed $\gamma \in \SL$, let $m$ be the smallest positive integer such that $\gamma T^m \gamma^{-1} \in G$. Suppose $f$ is modular with respect to $G$,  holomorphic, and non-zero on $\mathbb{H}$. Then
    $$ \frac{1}{2 \pi i} \left[ \log\left(f\right) \circ \gamma T^m \gamma^{-1} - \log\left(f\right) \right] = \mbox{\em ord}_{\gamma \infty}f.$$
\end{thm}

We apply this to $1-\lambda$ which satisfies the hypothesis of the theorem. Setting $\gamma = S$, we observe that $\text{ord}_{0}(1 - \lambda) = 1$. Therefore, we have
\begin{align*}
     \log(1-\lambda) \circ R^{-1} &= \log(1-\lambda) \circ ST^{2}S^{-1} \\
     &= \log(1- \lambda) + 2\pi i.
\end{align*}
This implies that
     $$
     (1-\lambda)^{\frac{1}{n}} \circ R = \xi_n^{-1} (1-\lambda)^{\frac{1}{n}},
     $$

where $\xi_n = e^{\frac{2 \pi i}{n}}$. Hence, $(1-\lambda)^{\frac{1}{n}}$ is invariant under the action of the group $\left< T^2,R^n \right>$. Moreover, it remains invariant under the action of any element of the form $AHA^{-1}$, where $A \in \overline{\Gamma}_0(2)$ and $H \in \left< T^2,R^n \right>$. This proves that $(1-\lambda)^{\frac{1}{n}}$ is a modular function over $N_{\overline{\Gamma}_0(2)} (T^2,R^n)$. The fact that it is a hauptmodul follows from:
$$
\left[\mathbb{C}\left((1-\lambda)^{\frac{1}{n}}\right) : \mathbb{C}(1-\lambda)\right] = n = \left[\overline{\Gamma}(2) : N_{\overline{\Gamma}_0(2)} (T^2,R^n) \right],
$$
which shows that $(1-\lambda)^{\frac{1}{n}}$ has degree $1$ over the modular curve 
$$
X(N_{\overline{\Gamma}_0(2)} (T^2,R^n)).
$$
Next, note that, for a positive integer $a$,
$$ 
(1-\lambda)^{\frac{a}{n}} \circ T = \frac{1}{(1-\lambda)^{\frac{a}{n}}}
$$
and 
$$ 
(1-\lambda)^{\frac{a}{n}} \circ R =  \xi_n^{-a} (1-\lambda)^{\frac{1}{n}}.
$$
It follows that $(1-\lambda)^{\frac{a}{n}}$ is a $\rho$-equivariant function on $\Gamma_0(2)$, where the image of $\rho$ is isomorphic to $D_{2n}$. Moreover, it is meromorphic at the cusps and  its derivative is nowhere vanishing on $\fH$. Therefore,  
$  \{ (1-\lambda)^{\frac{a}{n}}, \tau\}$ is a weight $4$ modular form over $\Gamma_0(2)$. More precisely, one can verify that
$$ 
\{ (1 - \lambda)^{r}, \tau\} = \frac{\pi^2}{2} (\theta_3 \theta_4)^4 + \frac{\pi^2}{2}r^2 (\theta_2^8),
$$
for all $r \in \mathbb{Q}^*$. 

In general, a modular solution to \eqref{4.3} in the dihedral case with $(m_1,m_2)=(n,4)$  is a rational function of $(1-\lambda)^{\frac{1}{n}}$ of degree
$$ d = \frac{n}{2} (n_1 -1) + n_2.$$

\begin{remark}
    As $(1 - \lambda) \circ S = \lambda$, we also have
    $$
    \{ \lambda^{r}, \tau\} = \frac{\pi^2}{2} (\theta_2 \theta_3)^4 + \frac{\pi^2}{2}r^2 (\theta_4^8).
    $$
\end{remark}

We compute a few of these rational functions as examples. In the following table, $t$ represents the hauptmodul $(1-\lambda)^{\frac{1}{n}}$.

\renewcommand{\arraystretch}{2} 
\begin{center}
\renewcommand{\arraystretch}{2} 
\begin{tabularx}{0.9\textwidth} { 
  | >{\centering\arraybackslash}X 
  | >{\centering\arraybackslash}X 
  | }
 \hline
 $(n, n_1, n_2, d)$ & $h(t)$ \\
 \hline
 $(2, 3, 1, 3)$  & $\displaystyle\frac{(t-1)^3}{3t^2 + 1}$ \\
\hline
 $(3, 3, 1, 4)$  & $\displaystyle\frac{(t-1)^3(t+1)}{2t^3 + 1}$ \\
 \hline
 $(n, 1, n_2, n_2), \; \text{gcd}(n, n_2) = 1$  & $t^{n_2}$ \\
\hline
\end{tabularx}
\end{center}
\ 

Now, let us consider the conjugate group $\Gamma = N_{\Gamma_0(2)}(T^n,R^2)$. For $n=1$, it is easy to see that this group is $\Gamma_0(4)$, whose hauptmodul can be taken as $\lambda(2 \tau)$ (since $\G(2)$ and $\G_0(4)$ are conjugate under the matrix $\binom{2\ \ 0}{0\ \ 1}$). Using similar methods to  the previous case, we find that
$$ \lambda(2 \tau)^{\frac{1}{n}} \circ T = \xi_n \lambda(2 \tau) \; \text{ and } \; \lambda(2 \tau)^{\frac{1}{n}} \circ R = \frac{1}{ \lambda(2 \tau)}. $$
Furthermore, for all $r \in \mathbb{Q}^*$, the Schwarz derivative satisfies:
$$ 
\{ \lambda(2 \tau)^{r}, \tau\} = 2\pi^2 r^2 (\theta_3 \theta_4)^4 + \frac{\pi^2}{8} \theta_2^8.
$$
In this case, the degree formula is given by:
$$ d = n_1 + \frac{n}{2}(n_2 -1).$$

\subsection*{The Octahedral Case}

We now consider the modular solutions $h$ to \eqref{4.3} whose projective representation has an image isomorphic to $S_4$. We focus on  the case  $$\Gamma = \ker{\overline{\rho}} = N_{\overline{\Gamma}_0(2)}(T^4,R^3),$$
with the other case being analogous.

The degree formula is given by
$$ d= 3n_1 + 4n_2 -6.$$
Furthermore, $\Gamma$ is not a congruence subgroup of the modular group because a torsion-free genus 0 congruence subgroup of the modular group has a most  index 60 in $\overline{\Gamma}(1)$ (see \cite{pams1}).

To construct an explicit hauptmodul, we start with a solution  $h$ that is $\rho-$equivariant for a representation $\rho$  of $\G_0(2)$.
One can easily verify that the matrices $\disp \binom{i\quad 0}{0\quad1}$ and $\disp \binom{1\  -1}{i\quad\   i}$  generate a group which is isomorphic to $S_4$. We claim that, up to applying a linear fractional transformation to $h$, we can assume that
  \[
  h \circ T = i h\ \  \text{and}\ \  h \circ R = \frac{h-1}{ih+i}.
  \]
 Indeed, by the second part of \propref{klein}, we can assume (up to a fractional linear transformation) that 
  $$
 \overline{\rho} ( \overline{\Gamma}_0(2))=\left\langle \binom{i\quad0}{0\quad 1}\ \, ,\, \binom{1\  -1}{i\quad\  i}\right\rangle.
 $$
Now, the homomorphism that sends $\overline{\rho}(T)$ to $\binom{i\ \ 0}{0\ \ 1}$ and $\overline{\rho}(R)$ to $\binom{1\  -1}{i\ \ \  i}$ is an automorphism of $\overline{\rho} \left( \overline{\Gamma}_0(2) \right) \cong S_4$. Since all automorphisms of $S_4$ are inner, we can  assume that $h$ is $\rho-$equivariant with 
\[\disp \rho(T)= \binom{i\quad 0}{0\quad1}\ \ \text{and}\ \  \disp \rho(R)=\binom{1\  -1}{i\quad\   i}.
\]

As for an explicit construction, $h$ is a hauptmodul for $N_{\overline{\Gamma}_0(2)}\left<T^4,R^3\right>$ which contains $x := (1 - \lambda)^{\frac{1}{3}}$. Moreover
$$
\left[\mathbb{C}(x) : \mathbb{C}(h) \right] = [S_4:D_6] = 4. 
$$
Hence, $x$ is a rational function of degree $4$ in $h$, expressed as
$$ 
x = \frac{a_4h^4 + a_3 h^3 + a_2 h^2 + a_1 h + a_0}{b_4h^4 + b_3 h^3 + b_2 h^2 + b_1 h + b_0},
$$
for some $a_i,b_i \in \mathbb{C}$, with $a_4b_4\neq 0$. These coefficients can be easily determined using the constraints:
\[
x(\infty) = 1\,,\  h(T \tau) = ih \,,\    x(T\tau) = \frac{1}{x}\,,
\]
\[  h(R\tau) = \frac{h-1}{ih+i} \, \text{ and } \, x(R\tau) = \xi_3^{-1} x.
\]
We  obtain
$$ x =  \frac{h^4 + 2 i \sqrt{3} h^2 + 1}{h^4 - 2 i \sqrt{3} h^2 + 1}.$$
Thus, up to a scalar multiple, we find
$$
h_{\pm} = \sqrt{\frac{ \sqrt{3}i(x+1) \pm 2 i \sqrt{1+x+x^2}}{x-1}},
$$
both of which are hauptmoduln for $N_{\Gamma_0(2)}\left<T^4,R^3\right>$. Additionally, one can verify that $h_{-} = 1/h_{+}$.

It is worth noting that the holomorphicity of $h_{\pm}$  on $\mathbb{H}$ follows from  the fact that $\lambda$ never takes the values $0$ or $1$ on the upper half-plane.

Using the same approach, one can construct a hauptmodul for $$ N_{\overline{\Gamma}_0(2)}(T^3,R^4)$$
as an algebraic expression of $\lambda(2 \tau)^{\frac{1}{3}}$.

\begin{remark}
    The same approach does not apply to $G = A_5$ because it is a simple group. This is the only case in this paper where explicit solutions are not provided.
\end{remark}

\subsection*{The Cyclic Case}

We recall some properties of the modular $\lambda-$function and introduce some new modular functions that will be used later on. Recall that $\lambda$ is $\rho-$equivariant on $\Gamma(1)$ as it satisfies \cite{rankin}
\begin{equation}\label{6.3}
    \lambda(\tau + 1) = \frac{\lambda(\tau)}{\lambda(\tau) -1} \; \text{ and } \; \lambda(-1/\tau) = 1 - \lambda(\tau).
\end{equation}
We also define the modular function
\begin{equation}\label{6.4}
    \omega_2(\tau) = 2^{12} q \prod_{n>0}(1+q^n)^{24} = 2^{12} \frac{\Delta(2 \tau)}{\Delta(\tau)},
\end{equation}
where $\Delta=\eta^{24}$ is the classical modular discriminant. It is not difficult to verify the identity
\begin{equation}\label{6.5}
     \omega_2 = 16 \frac{\lambda^2}{\lambda -1}.
\end{equation}
Moreover, the function $\omega_2$ is holomorphic and non-vanishing on $\mathbb{H}$. We define an $n$-th root of $\omega_2$. We first define 
$$ 
\log \omega_2 = 12\log(2) + 2 \pi i \tau + 24 \sum_{n > 0} \log(1+q^n)
$$
and then define
$ \omega_2^{\frac{1}{n}} := e^{\frac{1}{n} \text{log} \, \omega_2 }.$ 

When  $n$ is clear from context, we will denote $x := \omega_2^{-\frac{1}{n}}$. It transforms under $\Gamma_0(2)$ as follows \cite{yang}
\begin{equation}\label{6.6}
    x \circ T = \xi_{n}^{-1} x \; \text{ and } \; x \circ R = \xi_{n}^{-1} x.
\end{equation}

In this section, we  show that, except for $n=1$, the modular differential equation \eqref{4.3} has no modular solutions. Our approach involves classifying all $\rho-$equivariant functions for any $\rho: \overline{\Gamma}_0(2) \to \PG$, where Im$\,(\rho)\cong C_{2n}$ for $n \geq 2$. We begin by analyzing the various $\ker{\overline{\rho}}$ and studying their associated space of functions. 

When $n$ is even, the group $\G$ is given by:
$$ \Gamma = \left< T^{2n}, R^{2n}, T^{n+1}R^{-1}, [T, R]\right>.$$
These groups, studied in detail  in \cite{yang},  are denoted $\Phi_n(2n)$. The corresponding space of modular functions is 
$$ \BC(\Gamma) = \mathbb{C}(x,y),$$
where $x = \omega_2^{-\frac{1}{n}}$ and $y := (\frac{\lambda-2}{\lambda})\omega_2^{\frac{-1}{2n}}$ satisfy the hyperelliptic equation
$$ 
y^2 = x + 64 x^{n+1}.
$$
Under the action of $\Gamma_0(2)$, $y$ transforms as: 
\begin{equation}\label{6.7}
    y \circ T = - \xi_{2n}^{-1} y \; \text{ and } \; y \circ R =  \xi_{2n}^{-1} y.
\end{equation}

\begin{prop}
    Suppose $n > 1$ is an even integer. Then $h$ is a $\rho-$equivariant function whose projective representation has an image isomorphic $C_{2n}$ and a torsion-free kernel if and only if 
    $$ h  = \frac{\lambda - 2}{\lambda} \;\omega_2^{\frac{a}{2n} - \frac{1}{2}} R(\omega_2)$$
    for some rational function $R$ and some positive integer $a$ relatively prime to $2n$.
\end{prop}

\begin{proof}
    The ``only if" part is immediate using \eqref{6.3} and \eqref{6.6}. We now prove the converse.   First, observe that $$h \circ T = \xi_{2n}^a h, $$
    for some positive integer $a$ relatively prime to $2n$. From the group structure of $\Gamma$, we know that $\overline{\rho} (R) = \overline{\rho} (T)^{n+1}$. Hence 
    \begin{equation}
         h \circ T = \xi_{2n}^a h \  \text{ and } \  h \circ R = h \circ T^{n+1} = -\xi_{2n}^a h.
    \end{equation}
    The modular function $h$ lives in the function field $\mathbb{C}(x,y)$ of a hyperelliptic curve. Since $n > 1$, $h$ can be uniquely expressed as:
$$ h = f(x) + y g(x),$$
where $f$ and $g$ are rational functions.

Using identities (6.6) and (6.8) and the fact that  $h\circ T = \xi_{2n}^a h$, we deduce:
\begin{equation}
    \xi_{2n}^a (f(x) + yg(x))=h \circ T = f(x \circ T)  - \xi_{2n}^{-1} g(x \circ T) y.
\end{equation}
Similarly, from $h \circ R = h \circ T^{n+1} = -\xi_{2n}^a h$, we have:
\begin{equation}
    -\xi_{2n}^a (f(x) + yg(x))=h \circ R = f(x \circ R)  + \xi_{2n}^{-1} g(x \circ R) y.
\end{equation}

Combining these equations yields:
$$
f(x \circ T)  - \xi_{2n}^{-1} g(x \circ T) y = -f(x \circ R)  - \xi_{2n}^{-1} g(x \circ R) y.
$$
Since $x \circ R = x \circ T = \xi_n^{-1} x$, we conclude that $f = 0$ and equation (6.9) becomes
\begin{equation}
    g(\xi_n^{-1} x) = -\xi_{2n}^{a+1}g(x).
\end{equation} 

Now, define $\disp \Tilde{g}= g(x)/x^{\frac{n - a -1}{2}}$. It follows immediately that
$$
\Tilde{g} (\xi_n^{-1} x) = \Tilde{g} (x). 
$$
Thus, $\Tilde{g}(x)$ is a rational function of $x^n$ which we denote by $R(x^n)$. Therefore,
$$
h = y x^{\frac{n - a -1}{2}} R(x^n) = \frac{\lambda - 2}{\lambda} \omega_2^{\frac{a}{2n} - \frac{1}{2}} \Tilde{R}(\omega_2).
$$

\end{proof}

We proceed similarly for $n > 1$ odd. When
$$
\Gamma = \left< T^{2n}, R^{n}, T^{n+1}R^{-1}, [T, R]\right>,
$$
the space of modular functions is given by:
$$ 
\BC(\Gamma) = \mathbb{C}\left(\omega_2^{-\frac{1}{n}},\frac{\lambda-2}{\lambda}\right).
$$
Setting $x = \omega_2^{-\frac{1}{n}}$ and $y = (\frac{\lambda-2}{\lambda})$,  $x$ and $y$ verify the hyperelliptic equation
$$
y^2 = 1 + 64 x^{n}.
$$
We also deduce that a $\rho-$equivariant function, whose projective representation has an image isomorphic to $C_{2n}$,  has to be of the form
    $$ h  = \frac{\lambda - 2}{\lambda} \omega_2^{\frac{a}{2n} - \frac{1}{2}} R(\omega_2)$$
    where $R$ is a rational function and  $a$ is a positive integer relatively prime to $2n$.

Now, consider the case
$$
\Gamma = \left< T^{n}, R^{2n}, R^{n+1}T^{-1}, [T, R]\right>.
$$
Here, the space of modular functions is:
$$  \BC(\Gamma) = \mathbb{C}\left(\omega_2^{-\frac{1}{n}},\frac{\lambda-2}{\lambda} \omega_2^{-\frac{1}{2n}}\right).
$$
In this case, $h$ takes the form:
$$ h  = \frac{\lambda - 2}{\lambda} \omega_2^{\frac{a}{n} - \frac{1}{2}} R(\omega_2),$$
where $a$ and $n$ are relatively prime.

While this offers a broad class of $\rho$-equivariant functions, the next proposition will show that none of these functions can satisfy the modular differential equation \eqref{4.3}. The main limitation is that these functions lack non-vanishing derivatives.

\begin{lem}
    Let $r \in \mathbb{Q} \setminus \frac{1}{2}\mathbb{Z}$. An expression of the form 
    $$ h = \frac{\lambda -2}{\lambda} \omega_2^r \frac{P(\omega_2)}{Q(\omega_2)},$$
where $P$ and $Q$ are polynomials,    cannot solve the modular differential equation \eqref{4.3}.
\end{lem}

\begin{proof}

    We proceed by contradiction, assuming that $h'$ does not vanish. First, we can assume that $P(0) \ne 0$, $Q(0) \ne 0$ and that $P$ and $Q$ are relatively prime. Differentiating $h$, we obtain
        \begin{equation}\label{6.12}
          \scalebox{1}{$ h' = 16 \omega_2^{r-1} \lambda' \left( \frac{PQ[2(\lambda-1)^2 + r(\lambda-1)(\lambda-2)^2] + 16(P'Q - Q'P)\lambda^2(\lambda-2)^2}{(\lambda - 1)^3Q^2} \right)$}.
    \end{equation}

   Since $h'$ does not vanish,  then, as a rational function of $\lambda$, $P$ can have at most simple zeros at points in $\mathbb{C} - \{ 0,1\}$. Similarly, 
   \begin{equation*}
          \scalebox{1.1}{$\left(\frac{1}{h} \right)' = 16 \omega_2^{-r-1} \lambda' \left( \frac{PQ\lambda^2[2(\lambda-1)^2+r(\lambda-1)(\lambda-2)^2] + 16(P'Q-Q'P)\lambda^4(\lambda-2)^2}{(\lambda -2)^2(\lambda - 1)^3P^2} \right)$},
    \end{equation*}
   implies that $Q$, as a rational function of $\lambda$, can have at most simple zeros at points in $\mathbb{C} \setminus \{ 0,1,2\}$. 
    
{\bf Case 1: $Q(2) \ne 0$:}

   Write \eqref{6.12} as $h'=16 \omega_2^{r-1} \lambda'\,R(\lambda)$, where  $R = (\lambda-1)^\alpha\Tilde{R} (\tau)$ for some integer $\alpha$, and $\Tilde{R}(1)\in \BC^\times$. At $\lambda=0$, we have:
     $$
     R(0) = -\frac{P(0)}{Q(0)} \,(2-4r) \in \mathbb{C}^\times.
     $$
     Thus, $\Tilde{R}$ does not vanish at $\lambda=0$ or $\lambda=1$, and since $h'$ is nowhere vanishing, $\Tilde{R}$ does vanish anywhere in $\BC$.
     
     Now, a pole $a$ of $\Tilde{R}$ is a zero of $Q$, hence it is a pole of $h$ since $a \ne 2$. This implies  $a$ is a double pole for $h'$, and therefore it is a double pole for $\Tilde{R}$. Consequently,
     $$ 
     R = c\frac{(\lambda-1)^{\beta}}{Q(\omega_2)^2},
     $$
     for some integer $\beta \in \mathbb{Z}$ and a constant $c$. It follows  that 
     $$  
     h' = 16 c \omega_2^{r-1} \lambda' \frac{(\lambda-1)^{\beta}}{Q(\omega_2)^2}.
     $$
     Differentiating $h(\tau + 1) = \xi_{2n}^a h(\tau)$, we  obtain
     $$  
     16 \, c \, \xi_n^{r-1} \omega_2^{r-1} \frac{-\lambda'}{(\lambda -1)^2} \frac{(\lambda-1)^{-\beta}}{Q(\omega_2)^2} = 16 \, c \, \omega_2^{r-1} \lambda' \frac{(\lambda-1)^{\beta}}{Q(\omega_2)^2}.
     $$
    From this $-2-\beta=\beta$, so $\beta=-1$,  and 
     $$ 
     R = \frac{c}{(\lambda - 1) Q(\omega_2)^2}.
     $$
Substituting this back into the expression for $R$, we find
     $$ 
     PQ[2(\lambda-1)^2+r(\lambda-1)(\lambda-2)^2] + 16(P'Q-Q'P)\lambda^2(\lambda-2)^2 = c (\lambda - 1)^2.
     $$
     Since $r\notin\BZ$,  the degree in $\lambda$ of the left-hand side is $\text{deg}\,P +\text{deg}\, Q +3$, while the degree of the right-hand side is 2, leading to a contradiction.

     {\bf Case 2: $Q(2) = 0$:}
     
    If $Q$ vanishes at $\lambda=2$, then by \eqref{6.12} and the  non-vanishing of $(1/h)'$, the polynomial $Q$ has a zero of order at most 2 at $\lambda=2$. Noting that
     $$
     \omega_2'(\lambda) = 16 \frac{\lambda (\lambda -2)}{(\lambda-1)^2},
     $$
     we find that all derivatives of $Q$ with respect to $\lambda$ vanish at $2$. Consequently,  $R$ has a pole of order $2$ at $\lambda = 2$, and 
     $$  
     R = c (\lambda - 1)^{\alpha} \frac{(\lambda -2)^2}{Q(\omega_2)^2}.
     $$
    Similarly to  the previous case,  we obtain $\alpha=-1$ leading to a contradiction analogous to  the previous case.
\end{proof}

    \begin{remark}
        We observe that when $n = 1/2$,  there exists a solution
        $$ \left\{ \frac{\lambda -2}{\lambda} \omega_2^{1/2}, \tau \right\} =   2 \pi^2 \left( \theta_3 \theta_4 \right)^4 + 2 \pi^2 \frac{1}{16} \left(\theta_2\right)^8.$$
        This is a consequence of the fact that $\disp \frac{\lambda -2}{\lambda} \omega_2^{1/2}$ is a hauptmodul of $\Gamma_0(4)$.
    \end{remark}

\end{document}